\documentclass{amsart}

\usepackage{amsmath}
\usepackage{hyperref}
\usepackage{amsfonts,graphics,amsthm,amsfonts,amscd,latexsym, amssymb}

\DeclareFontFamily{U}{MnSymbolC}{}
\DeclareSymbolFont{MnSyC}{U}{MnSymbolC}{m}{n}
\DeclareFontShape{U}{MnSymbolC}{m}{n}{
    <-6>  MnSymbolC5
   <6-7>  MnSymbolC6
   <7-8>  MnSymbolC7
   <8-9>  MnSymbolC8
   <9-10> MnSymbolC9
  <10-12> MnSymbolC10
  <12->   MnSymbolC12}{}
\DeclareMathSymbol{\iprod}{\mathbin}{MnSyC}{'270}

\usepackage{epsfig}
\usepackage{flafter}
\usepackage{mathtools}
\usepackage{comment}
\usepackage{stmaryrd}

\usepackage{mathabx,epsfig}

\hypersetup{
    colorlinks=true,    
    linkcolor=blue,          
    citecolor=blue,      
    filecolor=blue,      
    urlcolor=blue           
}
\usepackage{tikz}
\usetikzlibrary{graphs,positioning,arrows,shapes.misc,decorations.pathmorphing}

\tikzset{
    >=stealth,
    every picture/.style={thick},
    graphs/every graph/.style={empty nodes},
}

\tikzstyle{vertex}=[
    draw,
    circle,
    fill=black,
    inner sep=1pt,
    minimum width=5pt,
]
\usepackage[position=top]{subfig}
\usepackage{amssymb}
\usepackage{color}

\calclayout

\usetikzlibrary{decorations.pathmorphing}
\tikzstyle{printersafe}=[decoration={snake,amplitude=0pt}]

\newcommand{\Cl}{\operatorname{Cl}}

\newcommand{\QQ}{\mathbb{Q}}
\newcommand{\ZZ}{\mathbb{Z}}

\newcommand{\CC}{\mathbb{C}}

\def\O#1.{\mathcal {O}_{#1}}			
\def\pr #1.{\mathbb P^{#1}}				
\def\af #1.{\mathbb A^{#1}}			
\def\ses#1.#2.#3.{0\to #1\to #2\to #3 \to 0}	
\def\xrar#1.{\xrightarrow{#1}}			
\def\K#1.{K_{#1}}						
\def\bA#1.{\mathbf{A}_{#1}}			
\def\bM#1.{\mathbf{M}_{#1}}				
\def\bL#1.{\mathbf{L}_{#1}}				
\def\bB#1.{\mathbf{B}_{#1}}				
\def\bK#1.{\mathbf{K}_{#1}}			
\def\subs#1.{_{#1}}					
\def\sups#1.{^{#1}}

\usepackage{tikz}
\usetikzlibrary{matrix,arrows,decorations.pathmorphing}

  \newtheorem{introthm}{Theorem}

  \newtheorem{theorem}{Theorem}[section]
  \newtheorem{lemma}[theorem]{Lemma}
  \newtheorem{proposition}[theorem]{Proposition}

\theoremstyle{remark}

\numberwithin{equation}{section}

\usepackage[all]{xy}

\begin{document}

\title{Orbifold Kähler-Einstein metrics on projective toric varieties}

\author[L.~Braun]{Lukas Braun}
\address{Mathematisches Institut, Albert-Ludwigs-Universit\"at Freiburg, Ernst-Zermelo-Strasse 1, 79104 Freiburg im Breisgau, Germany}
\email{lukas.braun@math.uni-freiburg.de}

\thanks{
The author is supported by the Deutsche Forschungsgemeinschaft (DFG) grant BR 6255/2-1. 
}

\maketitle

\begin{abstract}
In this short note, we investigate the existence of orbifold Kähler-Einstein metrics on toric varieties. In particular, we show that every $\QQ$-factorial normal projective toric variety allows an orbifold Kähler-Einstein metric. Moreover, we characterize $K$-stability of $\QQ$-factorial toric pairs of Picard number one in terms of the log Cox ring and the universal orbifold cover.
\end{abstract}

\setcounter{tocdepth}{1} 
\tableofcontents

\section{Introduction}
We work over the field $\CC$ of complex numbers.
In contrast to the case of negative or zero first Chern class - where Kähler-Einstein metrics are known to always exist - due to the confirmation of the Yau-Tian-Donaldson conjecture, we know that in the case of Fano manifolds, the existence of a Kähler-Einstein metric is equivalent to the algebraic notion of $K$-polystability~\cite{KE1, KE2, KE3, TianKE}. This purely smooth setting was extended in the last years  to the case of klt log Fano pairs $(X,\Delta)$ culminating in the analogous statement for such pairs~\cite[Thm.~1.6]{LXZ22}: the existence of a \emph{singular} Kähler-Einstein metric being equivalent to $K$-polystability of the pair $(X,\Delta)$. In the case of toric varieties, this is equivalent to the corresponding polytope having it's barycenter at the origin~\cite{WZ04,BB13,Ber16,BL22}.

It was conjectured in~\cite[Conj.~1]{Don12} that for non-$K$-stable Fano manifolds, a Kähler-Einstein metric with certain \emph{cone singularities} should exist. Partial results in this direction can e.g. be found in~\cite{OS15}, but there also have been found counterexamples to the original version of the conjecture~\cite[Thm.~1]{Sze}. In fact, the  counterexamples given are toric Gorenstein del Pezzo surfaces.
On the other hand, a modified version of the conjecture, see~\cite[Conj.~7.4]{BL22}, was proven in~\cite[Thm.~1.8]{LXZ22}, stating that for a log Fano pair $(X,\Delta)$, there exists $m$ and a $\QQ$-divisor $D$ in the linear system $\frac{1}{m}|-m(K_X+\Delta)|$, such that $(X,\Delta+D)$ is $K$-polystable. Before that, in~\cite[Thm.~7.10]{BL22}, the authors showed that in the toric case one can find a torus invariant boundary with these properties.

However, as Donaldson remarks~\cite{Don12}, the only singular metrics for which we know that "a great deal of the standard theory can be brought to bear`` are \emph{smooth orbifold metrics}. Singular (or weak) K\"ahler-Einstein metrics on orbifolds are smooth orbifold metrics, see~\cite{LT19}. To have an orbifold structure on our variety $X$, a necessary but not sufficient criterion is that $X$ has quotient singularities and $\Delta$ is snc on the smooth locus with so-called \emph{standard coefficients} of the form $1-\frac{1}{m}$. In the case of $\QQ$-factorial toric varieties, an orbifold structure exists if the boundary is torus invariant and has standard coefficients, see Proposition~\ref{prop:toricorbi}.

\subsection{Toric varieties with orbifold Kähler-Einstein metrics}
As mentioned above, a toric boundary with standard coefficients indeed provides an orbifold metric. Our first result says that one can always find such a boundary. 

\begin{introthm}
\label{thm:orbimetric}
Let $X$ be a normal projective toric variety. Then $X$ allows a toric boundary $\Delta$ \emph{with standard coefficients}, such that $(X,\Delta)$ is $K$-polystable. In particular, if $X$ is $\QQ$-factorial, it allows an orbifold Kähler-Einstein metric.
\end{introthm}

While this can be deduced from the existence of \emph{some} toric boundary $\Delta'$ with $(X,\Delta')$ $K$-polystable due to~\cite[Thm.~7.10]{BL22}, our proof of Theorem~\ref{thm:orbimetric} also provides an alternative purely convex geometric proof for~\cite[Thm.~7.10]{BL22}.

\subsection{$K$-stability in terms of the (log) Cox ring.}

Since all toric varieties have a polynomial Cox ring, the \emph{grading} by the class group alone must encode the $K$-polystability of a toric variety. In~\cite{BM21}, the authors introduced the notion of the \emph{log Cox ring} of a pair $(X,\Delta)$, which is the right object to study in this context, since it takes into account the boundary $\Delta$. 
For a toric orbifold boundary $\Delta$, the \emph{log class group} $\Cl(X,\Delta)$ is the quotient of orbifold Weil divisors ($\QQ$-divisors that become integral on orbifold charts) by linear equivalence. The log Cox ring is the associated divisorial algebra. It's spectrum $\hat{X}_\Delta$ - the \emph{log characteristic space} - allows a good quotient $\hat{X}_\Delta \to X$ by the diagonalizable group $H_{(X,\Delta)}:= \mathrm{Spec}\, \CC[\Cl(X,\Delta)]$ which ramifies over $\Delta_i$ with order $m_i$. In this setting, we have the following characterization of $K$-polystability:

\begin{introthm}
\label{thm:coxchar}
Let $X$ be a $\QQ$-factorial toric variety of Picard number one and dimension $n$ and let $\Delta=\sum_{\rho \in \Sigma(1)} \left(1-1/m_{\rho}\right) D_\rho$ be a toric orbifold boundary. Then the following are equivalent:
\begin{enumerate}
    \item $(X,\Delta)$ is $K$-polystable.
    \item The barycenter of $P^\vee_{-(K_X+\Delta)}=\mathrm{conv}(m_\rho u_\rho)_{\rho \in \Sigma(1)} \subseteq N_{\QQ}$ is $0$.
    \item The orbifold universal cover of $(X_{\mathrm{reg}},\Delta)$ is $(\mathbb{P}^n,\emptyset)$.
    \item There is a subgroup $\mathbb{Z} \leq \Cl(X,\Delta)$ such that $\mathrm{Spec}\, \CC[\ZZ] \cong \CC^*$ acts with weights $(1,\ldots,1)$ on $\hat{X}_\Delta$.
\end{enumerate}
\end{introthm}

Unfortunately, this characterization breaks down for higher Picard numbers. This is partly because the dual of a polytope (which is not a simplex) with barycenter at the origin may have it's barycenter away from the origin! Another reason is that there are many

\subsection*{Acknowledgements}
The author would like to thank Harold Blum, Yuchen Liu, Yuji Odaka, Chenyang Xu, and Ziquan Zhuang for helpful remarks on a previous version of this note.

\section{Preliminaries}

\subsection{Log pairs and their singularities}

Let $X$ be a normal variety and $\Delta$ be an effective $\QQ$-divisor. We call $(X,\Delta)$ a log pair if $K_X+\Delta$ is $\QQ$-Cartier. In case $0 \leq \Delta \leq 1$, we call $\Delta$ a boundary. Then for a log resolution $f\colon Y \to X$, we define the discrepancies of $K_X+\Delta$ to be the coefficients at exceptional prime divisors of the divisor $K_Y-f^*(K_X+\Delta)$. We say that $(X,\Delta)$ is a klt pair, if $\Delta<1$ and the discrepancies are greater than $-1$. We call $(X,\Delta)$ log Fano, if it is klt and $-(K_X+\Delta)$ is ample. Moreover, we say that $X$ is of klt type (Fano type), if there exists a boundary $\Delta$ with $(X,\Delta)$ klt (log Fano). 

\subsection{Toric Geometry}
\label{sec:toric}
We follow~\cite{CLS11}. Let $X$ be a toric variety with acting torus $T$. As usual, by $M$ and $N$ we denote the dual lattices of characters and one-parameter subgroups of $T$, respectively. Then $X=X_\Sigma$ for some polyhedral fan $\Sigma $ in $N_\QQ$. Every ray $\rho$ of $\sigma$ is associated with a $T$-invariant prime divisor $D_\rho$, and these generate the group of $T$-invariant Weil divisors. We denote the primitive ray generators by $u_\rho$. Elements $m \in M_\QQ$ define $T$-invariant $\mathbb{Q}$-principal divisors $D_m$ in the following way:
$$
D_m= \sum_{\rho \in \Sigma} - \langle m, u_\rho \rangle \, D_\rho.
$$
Since the Picard group of affine toric varieties $X_\sigma$ is trivial, consequently Cartier divisors on $X_\Sigma$ are just given by collections $(m_\sigma)_{\sigma \subseteq \Sigma}$ such that $\langle m_\sigma, u_\rho \rangle = \langle m_{\tau}, u_\rho \rangle$ whenever $\rho$ is a common ray of $\sigma$ and $\tau$~\cite[Thm.~4.2.8]{CLS11}. Obviously, it suffices to specify $m_\sigma$ for the maximal cones of $\Sigma$. The situation gets even simpler if we consider \emph{ample} divisors. For those, the $m_\sigma$ are pairwise distinct and form the vertices of a convex polytope $P_D \subseteq M_\QQ$~\cite[Cor.~6.1.16]{CLS11}. Moreover, the normal fan of $P_D$ is $\Sigma$ and the vertices of the dual polytope $P_D^\vee$ are supported on the rays of $\Sigma$. If we denote such a vertex supported on $\rho$ by $v_\rho$, then the value $a_\rho$ of $D$ at $D_\rho$ is the rational number satisfying $a_\rho v_\rho = u_\rho$. In particular, $P_D^\vee$ is a lattice polytope if and only  if the $a_\rho$ are of the form $1/k$ with $k \in \ZZ$.

A $T$-invariant canonical divisor on a toric variety is given by $K_{X_\Sigma}=-\sum_{\rho \in \Sigma} D_\rho$. So for a boundary $\Delta=\sum_{\rho \in \Sigma} a_\rho D_\rho$, the vertices of the polytope $P_{-(K_X+\Delta)}^\vee$ are given by $v_\rho=\frac{1}{1-a_\rho}u_\rho$.

Depending on the needs, people in toric geometry either work with the polytope $P_D \subseteq M_\QQ$, the dual $P_D^\vee \subseteq N_\QQ$, or both. $K$-(poly/semi)stability is equivalent to the barycenter of $P_D \subseteq M_{\QQ}$ lying at the origin~\cite{WZ04,BB13,Ber16,BL22}.

\section{Proofs of the main statements}

We start with observing that a toric boundary with standard coefficients on a $\QQ$-factorial (not necessarily complete) toric variety induces an orbifold structure, a statement that should be very well known to experts but which we haven't found in the literature. 

\begin{proposition}
\label{prop:toricorbi}
Let $X$ be a normal toric $\QQ$-factorial variety and $\Delta$ a toric boundary with standard coefficients. Then the pair $(X,\Delta)$ is an orbifold.
\end{proposition}

\begin{proof}
Let $X=X_\Sigma$. We have to show that locally $(X,\Delta=\sum_{\rho \in \Sigma(1)} (1-1/m_\rho)D_\rho)$ is a finite quotient ramifying over $D_\rho$ of order $m_\rho$. For the toric canonical divisor $K_X$, the formula from Section~\ref{sec:toric} gives
$$
-(K_X+\Delta)=\sum_{\rho \in \Sigma(1)} \frac{1}{m_\rho} D_\rho.
$$
Take a maximal cone $\sigma \subseteq \Sigma$ with extremal rays $\rho_1,\ldots, \rho_n$ and corresponding primitive ray generators $v_1,\ldots,v_n$. Here $n=\mathrm{dim}(X)$ by $\QQ$-factoriality. The toric log Cox construction, see~\cite[Section~3.1]{BM21}, is given by the map of lattices $\ZZ^n \to N; e_i \mapsto v_i$, i.e. by the multiplication with the matrix $P$ having the $v_i$ as columns. In particular, the grading of the log Cox ring $\CC[x_1,ldots,x_n]$ is given by the matrix $Q$ Gale dual to $P$~\cite[Section~2.2]{ADHL15}. This is a toric morphism from a smooth variety - in particular a finite quotient by a finite abelian group - which ramifies over $D_\rho$ of the right order $m_\rho$, see~\cite[Chapter~3.3]{CLS11}. 
\end{proof}

\begin{proof}[Proof of Theorem~\ref{thm:orbimetric}]
Let $X=X_\Sigma$ be a normal projective toric variety. Choose some ample toric $\QQ$-divisor $L=\sum a_\rho D_\rho$. Since $L$ is ample, it corresponds to a full dimensional rational convex polytope 
$$
P_L = \{ u \in M_\QQ ~|~ \langle u, v_\rho \rangle \geq -a_\rho  \, \forall \rho \in \Sigma(1) \}  \subseteq M_\QQ,
$$
not necessarily containing the origin. We denote by $u_{P_L}$ the barycenter of $P_L$. Since $P_L$ is full dimensional and convex, we have  $u_{P_L} \in P_L^\circ$. Now denote by $P'$ the translation of $P_L$ by $-u_{P_L}$:
$$
P':=P_L -u_{P_L} = \{ u \in M_\QQ ~|~ \langle u, v_\rho \rangle \geq -a_\rho + \langle u_{P_L},v_\rho \rangle \,  \forall \rho \in \Sigma(1) \}.
$$
The polytope $P'=P_{L'}$ has it's barycenter at the origin and corresponds to an ample $\QQ$-divisor $L':=\sum (a_\rho - \langle u_{P_L},v_\rho \rangle)D_\rho$, which, since $u_{P_L}$ was in the interior of $P_L$, now is effective and fully supported on $\sum D_\rho$. That is, the coefficients $b_\rho:=(a_\rho - \langle u_{P_L},v_\rho \rangle)$ are strictly positive rational numbers. We write $b_\rho=p_\rho/q_\rho$ with natural numbers $p_\rho$ and $q_\rho$ and denote $r:=\mathrm{lcm}(p_\rho)_{\rho \in \Sigma(1)}$ and $m_\rho:=l \cdot b_\rho^{-1}$. Then scaling $P_{L'}$ by $1/l$ yields another polytope
$$
P'':=\frac{1}{l}P'=\{ u \in M_\QQ ~|~ \langle u, v_\rho \rangle \geq -1/m_\rho \,  \forall \rho \in \Sigma(1) \},
$$
still having it's barycenter at the origin and corresponding to the ample divisor $L''$. Writing
$$
L''=\sum_{\rho \in \Sigma(1)} \frac{1}{m_\rho} D_\rho = -( K_X + \underbrace{\sum_{\rho \in \Sigma(1)} (1-\frac{1}{m_\rho})D_\rho}_{=: \Delta} ) 
$$
yields a $K$-polystable pair $(X,\Delta)$ with $\Delta$ having standard coefficients. This proves the first statement of the Theorem. The second statement then follows from Proposition~\ref{prop:toricorbi} and e.g. the considerations in~\cite{LT19}.
\end{proof}

The following is another easy but useful observation concerning barycenters of dual simplices, that we haven't found in the literature either.

\begin{lemma}
\label{le:simplexdual}
Let $P \subseteq \QQ^n$ be a simplex. Then $b_P=0$ if and only if $b_{P^\vee}=0$.
\end{lemma}

\begin{proof}
Since $(P^\vee)^\vee=P$, we only have to prove that $b_{P^\vee}=0$ if $b_P=0$. So, assuming $b_P=0$ and applying a change of basis, we are in the situation that the vertices of $P$ are $a_i=e_i$ for $1\leq i \leq n$ (where  $(e_i)_i$ is the standard basis), and $a_{n+1}=-\sum e_i$.

The facets of $P^\vee$ are given by the $n+1$ hyperplanes $\{x_i=-1\}$ ($1\leq i \leq n$) and $\{\sum x_i=1 \}$. Thus in the dual basis $(e^i)_i$,  the vertices of $P^\vee$ are given by
$$
b_k:=n e^k + \sum_{k \neq i=1}^n - e^i
\,
\mathrm{~for~}
\,
1\leq k \leq n,
\, 
\mathrm{~and~}
\,
b_{k+1}:= \sum_{i=1}^n - e^i.
$$
Thus $(n+1)b_{P^\vee}=\sum_{k=1}^{n+1} b_k=0$ and the claim is proven.
\end{proof}

\begin{proof}[Proof of Theorem~\ref{thm:coxchar}]
The equivalence of $(1)$ and $(2)$   follows from~\cite[Thm.~1.2]{BB13} and Lemma~\ref{le:simplexdual}.

Now assume that $(2)$ holds, i.e. the barycenter of $P:=\mathrm{conv}(m_\rho u_\rho)$ is zero, where $\Delta=\sum (1-1/m_\rho) D_\rho$ and $u_\rho$ are the primitive lattice generators of $\Sigma_X$. Choose some numbering $\rho_0,\ldots,\rho_n$ of the columns of $\Sigma_X$. Then the matrix with columns $m_{\rho_1} u_{\rho_1},\ldots,m_{\rho_n} u_{\rho_n} $ yields a lattice homomorphism (and a vector space isomorphism) which - since $b_p=0$ -  maps the cones of the fan $\Sigma_{\mathbb{P}^n}$ to the cones of $\Sigma_X$ and thus by~\cite[Thm.~3.3.4]{CLS11} yields a toric morphism $\mathbb{P}^n \to X$. This morphism ramifies over $D_\rho$ exactly with order $m_\rho$ and thus due to (the log version of)~\cite[Thm.~12.1.10]{CLS11} corresponds to the orbifold universal cover of $(X,\Delta)$. So $(3)$ follows from $(2)$. 

Again by~\cite[Thm.~12.1.10]{CLS11}, since $\pi_1^{\mathrm{orb}}(X_{\mathrm{reg}},\Delta)=N/N_{(\Sigma,\Delta)}$, this group is a quotient of $\mathrm{Cl}(X,\Delta)$ by some subgroup $H$, such that $\hat{X}_\Delta /\!/ H = \Tilde{X}_\Delta$ is the orbifold universal cover of $(X,\Delta)$.
Now this group is isomorphic to $\ZZ$ and acts with weights $(1,\ldots,1)$ if and only if $\Tilde{X}_\Delta=(\mathbb{P}^n,\emptyset)$. So $(3)$ and $(4)$ are equivalent.

Finally assume that $(3)$ holds. Then the covering $\mathbb{P} \to X$ is toric and again by~\cite[Thm.~3.3.4]{CLS11} yields a lattice homomorphism (and a vector space isomorphism) mapping the cones of  $\Sigma_{\mathbb{P}^n}$ to the cones of $\Sigma_X$. This homomorphism maps the barycenter of $P^\vee_{\mathbb{P}^n}$ - which is the origin - to the barycenter of $P^\vee_{-(K_X+\Delta)}$, which therefore is the origin. Thus $(2)$ follows from $(3)$ and the claim is proven. 
\end{proof}

\bibliographystyle{habbrv}
\bibliography{bib}

\end{document}